\documentclass[runningheads, a4paper]{llncs}
\usepackage{amsmath, amssymb, graphicx,algorithmic, algorithm,booktabs} 
\newcommand{\reals}{\mathbb{R}}
\newcommand{\gtri}{g^{*}}
\newcommand{\gtwo}{g^{\rm quad}}
\newcommand{\gmom}{g^{\rm mom}}
\newcommand{\gsos}{g^{\rm sos}}
\newcommand{\sdpexact}{{\sf SDP-EXACT}}

\usepackage{color}
\definecolor{dkmag}{rgb}{0.5,0,0.5}

\begin{document}
\mainmatter	
\title{A QCQP Approach to Triangulation} 
\titlerunning{A QCQP Approach to Triangulation}
\authorrunning{Chris Aholt, Sameer Agarwal and Rekha Thomas}
\author{Chris Aholt\inst{1}\and Sameer Agarwal\inst{2} \and Rekha Thomas\inst{1}}
\institute{University of Washington \and Google Inc.}

\maketitle

\begin{abstract}
Triangulation of a three-dimensional point from $n\ge 2$ two-dimensional images can be formulated as a quadratically constrained quadratic program. We propose an algorithm to extract candidate solutions to this problem from its semidefinite programming relaxations.
We then describe a sufficient condition and a polynomial time test for certifying when such a solution is optimal. This test has no false positives. Experiments indicate that false negatives are rare, and the algorithm has excellent performance in practice. We explain this phenomenon in terms of the geometry of the triangulation problem.

\end{abstract}
\section{Introduction}
\label{sec:introduction}
We consider the problem of triangulating a point $X \in \reals^{3}$ from $n\ge2$ noisy image projections. This is a fundamental problem in multi-view geometry and is a crucial subroutine in all structure-from-motion systems~\cite{hartley-zisserman-2003}.

Formally, let the point $X\in\mathbb{R}^3$ be visible in $n\ge 2$ images. Also let  $P_{i} \in \mathbb{R}^{3\times 4}$ be a projective camera and $x_{i} \in \reals^{2}$ be the projection of $X$ in image $i$, i.e,
\begin{align}
x_{i} = \Pi P_{i} \tilde{X},\ \forall\ i = 1,\ldots,n,
\end{align}
where, using {\tt MATLAB} notation, $\tilde{X} = \begin{bmatrix}X; 1 \end{bmatrix}$ and $\Pi\begin{bmatrix}u;v;w\end{bmatrix} = \begin{bmatrix} u/w; v/w\end{bmatrix}$. 

Given the set $\{x_{i}\}$ of noise free projections, it is easy to determine $X$ using a linear algorithm based on singular value decomposition (SVD)~\cite{hartley-zisserman-2003}. However, in practice we are given $\widehat x_i = x_i + \eta_i$, where $\eta_i$ is noise, and there are no guarantees on the quality of the solution returned by the linear algorithm. 

For simplicity, we assume that  $\eta_i \sim \mathcal{N}(0,\sigma I)$. 
Then the {\em triangulation
problem} is to find the maximum likelihood estimate of $X$ given the noisy observations $\{\widehat{x}_{i}\}$. Assuming such a point $X$ always exists, this is equivalent to solving:
\begin{align}
	\arg \min_X \sum_i^n \| \Pi P_i \tilde{X} -  \widehat{x}_i\|^2. \label{eq:rational}
\end{align}
Here and in the rest of the paper we will ignore the constraint that the point $X$ has positive depth in each image. The above optimization problem is an instance of fractional programming which is in general hard~\cite{freund-jarre-2001}. An efficient and optimal solution of~\eqref{eq:rational} is the subject of this paper.

For $n=2$, Hartley \& Sturm showed that~\eqref{eq:rational} can be solved optimally in polynomial time~\cite{hartley-sturm-1997}. For $n=3$, a Gr\"obner basis based algorithm for~\eqref{eq:rational} was proposed in \cite{stewenius-et-al-2005}. This algorithm relies on the observation that generically, the optimal solution to~\eqref{eq:rational} is one among the $47$ solutions to a certain system of polynomial equations.  This Gr\"obner basis method is not usefully extendable to higher $n$ 
and efficient optimal triangulation for $n \geq 4$ views has remained an unsolved problem. Other approaches either use the linear SVD based algorithm as initialization followed by non-linear refinement which lead to locally optimal solutions with no guarantees on the run time complexity~\cite{hartley-zisserman-2003}, or optimal algorithms whose worst case complexity is exponential in $n$~\cite{kahl-et-al-2008,kahl-henrion-2007,lu-hartley-2007}.

We present a new triangulation algorithm for $n \ge 2$ views. Based on semidefinite programming, the algorithm in polynomial time either determines a globally optimal solution to~\eqref{eq:rational} or informs the user that it is unable to do so. Theoretically, the operating range (in terms of image noise) of the algorithm is limited and depends on the particular configuration of the cameras. In practice our method computes the optimal solution in the vast majority of test cases.  In the rare case that optimality cannot be certified, the algorithm returns a solution which can be used as an initializer for nonlinear least squares iteration.

The paper is organized as follows.  In Section~\ref{sec:quadratic} we formulate triangulation as a constrained quadratic optimization problem. We present semidefinite relaxations to this problem in Section~\ref{sec:semidefinite}. We propose our triangulation algorithm in  Section~\ref{sec:analysis} and analyze its performance. We also provide theoretical explanation for why the algorithm works. Section~\ref{sec:experiments} presents experiments on synthetic and real data and we conclude in  Section~\ref{sec:discussion} with a discussion.

{\bf Notation.} We will use {\tt MATLAB} notation to manipulate matrices and vectors, e.g., $A[1:2,2:3]$ refers to a $2 \times 2$ submatrix of $A$. $P= \{P_{1}, \ldots, P_{n}\}$ denotes the set of cameras. $x = \left[x_{1};\dots;x_{n}\right]$ denotes a vector of image points, one in each camera, and  
$\widehat{x} = \left[\widehat{x}_{1};\dots;\widehat{x}_{n}\right]$ denotes the vector of image observations. Both $x$ and $\widehat{x}$ lie in $\mathbb{R}^{2n}$. If $y \in \reals^{m}$ is a vector, then $\tilde{y} = \begin{bmatrix} y; 1\end{bmatrix}$ is the homogenized version of $y$. The inner product space of $k\times k$ real symmetric matrices
is denoted $\mathcal{S}^k$ with the inner product $\langle A,B\rangle =
\sum_{1\le i,j\le k}A_{ij}B_{ij}$. The set $\mathcal{S}^{k}_{+} \subseteq \mathcal{S}^{k}$ denotes the closed convex cone of positive semidefinite matrices. We write $A\succ 0$ (resp. $A\succeq 0$) to mean that $A$ is positive definite (resp. positive semidefinite).


\section{Triangulation As Polynomial Optimization}
\label{sec:quadratic}
With an eye towards the future, let us re-state the triangulation problem~\eqref{eq:rational} as the constrained optimization problem
\begin{align}
\underset{x_{1},\dots,x_{n},X}{\arg\min}\  \sum_{i} \| x_{i} - \hat{x}_{i}\|^{2},\quad \text{s.t.} \ x_{i} = \Pi P_{i} \tilde{X},\quad \forall i = 1,\ldots,n.
\label{eq:rational-constrained}
\end{align}
In this formulation, the constraints state that each $x_{i}$ is the projection of $X$ in image $i$.  Let us now denote the feasible region for this optimization problem by 
\begin{align}
V_{P} = \left\{ x \in \reals^{2n}\ |\ \exists X \in \reals^{3} \text{ s.t. } x_{i} = \Pi P_{i} \tilde{X},\  \forall i = 1,\ldots,n\right\}.
\end{align}
For any $x \in V_{P}$, we can recover the corresponding $X$ using the SVD based algorithm~\cite{hartley-zisserman-2003}. Now we can re-state~\eqref{eq:rational-constrained} purely in terms of $x$
\begin{align}
\arg\min_{x}\   \| x - \hat{x}\|^{2},\quad
\text{s.t.}  \  x \in V_{P}.
\label{eq:abstract}
\end{align}

Let $\overline{V}_{P} \supseteq V_{P}$ be the closure of $V_{P}$, meaning that $\overline{V}_{P}$ contains all the limit points of $V_{P}$.  Then consider the following optimization problem:
\begin{align}
\arg\min_{x}\   \| x - \hat{x}\|^{2},\quad
\text{s.t.}  \  x \in \overline{V}_{P}.
\label{eq:abstract-closure}
\end{align}

The objective function in these two optimization problems is the squared distance from $\widehat{x}$ to the sets $V_{P}$ and $\overline{V}_{P}$ respectively. Since $\overline{V}_{P}$ is the topological closure of $V_{P}$ it can be shown that any solution $x^{*}$ to~\eqref{eq:abstract-closure} which is not in $V_{P}$ 
is arbitrarily close to a point in $V_{P}$, and the optimal objective function values for~\eqref{eq:abstract} and~\eqref{eq:abstract-closure} are the same. Thus, solving~\eqref{eq:abstract-closure} is essentially equivalent to solving~\eqref{eq:abstract}.

The set $V_{P}$ is a quasi-projective variety. A {\em variety} is the zero set of a finite set of polynomials, and a {\em quasi-projective variety} is the set difference of two varieties. Therefore, 
$\overline{V}_{P}$ is also a variety~\cite{shafarevich-1988}. Heyden \& {\AA}str{\"o}m~\cite{heyden-astrom-1997} show that
\begin{align}
\overline{V}_{P} = \left\{ x \in \reals^{2n}\   
\begin{array}{|l}   
\ f_{ij}(x) =0,\ 1\leq i < j \leq n\\
\ t_{i,j,k}(x) = 0,\ 1 \leq i < j < k \leq n 
\end{array}
\right \}.
\end{align}
Here $f_{ij}(x) =  \tilde{x}_{i}^{\top} F_{ij} \tilde{x}_{j} = 0$ are the bilinear/quadratic epipolar constraints, where $F_{ij}\in \reals^{3 \times 3}$ is the fundamental matrix for images $i$ and $j$.    The second set of constraints $t_{ijk}(x_{i},x_{j},x_{k}) = 0$ are the  the trilinear/cubic constraints defined by the trifocal tensor on images $i,j$ and $k$.

At the risk of a mild abuse of notation, we will also use $F_{ij}$ to denote a $(2n+1) \times (2n+1)$ matrix such that $f_{ij}(x) = \tilde{x}^{\top} F_{ij} \tilde{x}$. The construction of this matrix involves embedding two copies of the $3 \times 3$ fundamental matrix  (with suitable reordering of the entries) in an all zero $(2n+1) \times (2n+1)$ matrix.

Now, let $W_{P}$ be the quadratic variety
\begin{align}
\label{eq:quadratic-variety}
W_{P} = \left\{ x \in \reals^{2n}\ | \right . & \left . \tilde{x}^{\top} F_{ij}\tilde{x} = 0,\  1 \leq i < j \leq n \right\}.
\end{align}

For $n=2$, since there are no trilinear constraints $W_{P} = \overline{V}_{P}$ but for $n \geq 3$, in general $W_{P} \supseteq \overline{V}_{P}$.   For $n \ge 4$, Heyden \& {\AA}str{\"o}m show that if the camera centers are not co-planar then $W_{P} = \overline{V}_{P}$~\cite{heyden-astrom-1997}.  Note that for $n=3$, the camera centers are always co-planar.
Therefore, for $n=2$,  and for $n\geq 4$ when the camera centers are non-co-planar, we can just optimize over the quadratic variety $W_{P}$:
\begin{align}
\arg\min_{x} \  \| x - \hat{x}\|^{2}\quad
\text{s.t.}  \ x \in W_{P}.
\label{eq:abstract2}
\end{align}

However, we cannot just ignore the co-planar case as a degeneracy since  it is a common enough occurrence, e.g., an object rotating on a turntable in front of a fixed camera. If all the camera centers lie on a plane $\pi_{P}$, then solving~\eqref{eq:abstract2} instead of~\eqref{eq:abstract-closure} can result in spurious solutions, i.e. a projection vector $x^{*}$ for which there is no single point $X^{*} \in \mathbb{R}^{3}$ that projects to $x^{*}_{i}$ for each image $i$. This can happen if each $x^{*}_{i}$ lies on the image of the plane $\pi_{P}$ in image $i$. It is easy to reject such a spurious $x^{*}$ by checking if it satisfies the trilinear constraints.

From here on, we will focus our attention on solving~\eqref{eq:abstract2} and in Section~\ref{sec:experiments} we will show that solving~\eqref{eq:abstract2} instead of~\eqref{eq:abstract-closure} is not a significant source of failures.

Let us now define the polynomial 
\begin{align}
g(x) &= \| x - \widehat{x} \|^{2} 
= \tilde{x}^{\top}
\begin{bmatrix} I  & -\widehat x\\
-\widehat x^{\top} & \|\widehat x\|^2\end{bmatrix}
\tilde{x}
=  \tilde{x}^{\top}G\tilde{x},
\end{align}
where $I$ is the $2n\times 2n$ identity matrix. Observe that $G \in \mathcal{S}^{2n+1}_{+}$, and the Hessian of $g$ is $\nabla^2g=2I$. Similarly, let  $F_{ij} = \begin{bmatrix} H_{ij} & b_{ij}\\ b_{ij}^{\top} & \beta_{ij}\end{bmatrix}$,
where $H_{ij}\in\mathcal{S}^{2n}$, $b_{ij}\in\mathbb{R}^{2n}$, and $\beta_{ij}\in\mathbb{R}$.  Then $\nabla^{2}f_{ij}=2H_{ij}$. We re-write~\eqref{eq:abstract2} as the quadratically constrained quadratic program (QCQP)
\begin{equation}
\label{eq:matrix-polynomial}
\begin{array}{rl}
\displaystyle\arg\min_{x\in\reals^{2n}}& \tilde{x}^{\top}G\tilde{x}\quad
\text{s.t.} \quad \tilde{x}^{\top}F_{ij}\tilde{x} = 0,\   1 \leq i < j \leq n.
\end{array}
\end{equation}


\section{Semidefinite Relaxation}
\label{sec:semidefinite}

We re-write~\eqref{eq:matrix-polynomial} as the following rank constrained semidefinite program (SDP).
\begin{equation}
\label{triangulation with matrices}
\begin{array}{rl}
\displaystyle\arg\min_{Y}& \langle G,Y\rangle \\
\text{s.t.} & \langle F_{ij}, Y\rangle=0,\ 1\le i<j\le n,\\
& \langle E, Y\rangle = 1,\\
& Y\in\mathcal{S}^{2n+1}_{+},\\
& {\rm rank}(Y)=1.
\end{array}
\end{equation}
Here $E \in \mathcal{S}^{2n+1}$ is an all zero matrix except for its bottom right entry which equals one.
The problems~\eqref{eq:matrix-polynomial} and \eqref{triangulation with matrices} are equivalent: $x$ is feasible (optimal) for \eqref{eq:matrix-polynomial} if and only if $Y=\tilde{x}\tilde{x}^{\top}$ is feasible (optimal) for~\eqref{triangulation with matrices}.

  

Solving rank constrained semidefinite programs is NP-hard~\cite{vandenberghe-boyd-1996}. Dropping the 
rank constraint gives the primal semidefinite program
\begin{equation}
\label{primal sdp}
\begin{array}{rl}
\displaystyle\arg\min_{Y}& \langle G,Y\rangle \\
\text{s.t.} & \langle F_{ij}, Y\rangle=0,\ 1\le i<j\le n,\\
& \langle E, Y\rangle = 1,\\
& Y\in\mathcal{S}^{2n+1}_{+}.
\end{array}
\end{equation}
The dual of this primal semidefinite program is
\begin{equation}
\label{dual sdp}
\begin{array}{rl}
\underset{\lambda_{ij},\rho}{\arg\max} & \rho \\
\text{subject to } & G + \sum \lambda_{ij} F_{ij} -\rho E\succeq 0,\\
& \lambda_{ij}, \rho\in\mathbb{R},\ 1\le i<j\le n.
\end{array}
\end{equation}

The primal SDP~\eqref{primal sdp} is also known as the first {\em moment} relaxation
and its dual SDP~\eqref{dual sdp} is known as the first {\em sum of squares} relaxation.
They are instances of a general hierarchy of semidefinite relaxations for polynomial optimization problems~\cite{laurent-2009}. Problem \eqref{dual sdp} is also the Lagrangian dual of~\eqref{eq:matrix-polynomial}.

The remainder of this paper is dedicated to the possibility that solving the triangulation problem is equivalent to solving these semidefinite relaxations. Let us denote by $\gtri,\gtwo,\gmom,\gsos$, the optimal solutions to the optimization problems~\eqref{eq:abstract-closure},~\eqref{eq:matrix-polynomial},~\eqref{primal sdp},~\eqref{dual sdp} respectively. Then the following lemmas hold.

\begin{lemma}
\label{lem:tri-quad-inequality}
For all $n$, $\gtri \geq \gtwo$. For $n=2$, or $n\geq4$ with non-co-planar cameras, $\gtri = \gtwo$.
\end{lemma}
\begin{proof}
The claim follows from the discussion in Section~\ref{sec:quadratic} after the definition of $W_{P}$. $\qed$
\end{proof}
\begin{lemma}
\label{lem:mom-relax-inequality}$\gtwo \geq \gmom$.
\end{lemma}
\begin{proof}
This is true because~\eqref{primal sdp} is a relaxation of~\eqref{eq:matrix-polynomial}.
$\qed$
\end{proof}

\begin{lemma}
\label{lem:zero-duality}
For all $n$, $\gmom = \gsos$.  Moreover, there exist optimal $Y^{*},\lambda^{*}_{ij}$ and $\rho^{*}$ that achieve these values.
\end{lemma}
\begin{proof}
The inequality $\gmom \geq \gsos$ follows from weak duality. Equality, and the existence of $Y^{*},\lambda^{*}_{ij}$ and $\rho^{*}$ which attain the optimal values follow if we can show that the feasible regions of both the primal and dual problems have nonempty interiors \cite[Theorem 3.1]{vandenberghe-boyd-1996} (also known as Slater's constraint qualification). 

For the primal problem, let $x\in\mathbb{R}^{2n}$ be any feasible
point for the triangulation problem~\eqref{eq:abstract2} (such a feasible point always exists) and let $D={\rm diag}(1,\ldots,1,0)\in\mathcal{S}^{2n+1}$. It is easy to show that 
$Y= \tilde{x}\tilde{x}^{\top} + D$ is positive definite and primal feasible. 
For the dual problem, take $\lambda_{ij}=0$ and $\rho=-1$ and verify $G+E\succ 0$. \qed
\end{proof}


\section{The Algorithm and its Analysis}
\label{sec:analysis}

We propose Algorithm~\ref{alg:triangulation} as a method for triangulation.
\begin{algorithm}[h!]
\caption{Triangulation}
\label{alg:triangulation}
\begin{algorithmic}[1]
\REQUIRE Image observation vector $\widehat{x} \in\reals^{2n}$ and the set of cameras $P=\{P_{1},\ldots,P_{n}\}$.
\STATE Solve the primal~\eqref{primal sdp} and dual~\eqref{dual sdp} SDPs to optimal $Y^{*}, \lambda_{ij}^{*}$ and  $\rho^{*}$.
 \STATE $x = Y^{*}[1:2n,2n+1]$ (i.e., $x$ is the last column of $Y^{*}$ without its last entry)
\STATE Use the SVD algorithm to determine a world point $X\in\mathbb{R}^3$ from $x$.\label{alg-line:extract}
\IF{$I+\sum\lambda_{ij}^{*}H_{ij} \succ 0$} \label{alg-line:definite}
\IF{$n=2$,  or $n \geq 4$ and the cameras $P$ are non-co-planar,}
\STATE Return ({\sf OPTIMAL}, $X$).
\ENDIF
\IF {$x_{i}^{} = \Pi P_{i}\tilde{X}^{}\  \forall i=1,\ldots,n$}
\STATE Return ({\sf OPTIMAL}, $X^{}$)
\ENDIF
\ENDIF
\STATE Return ({\sf SUBOPTIMAL}, $X^{}$).
\end{algorithmic}
\end{algorithm}

\subsection{Correctness}

\begin{theorem}
Algorithm~\ref{alg:triangulation} terminates in time polynomial in $n$.
\end{theorem}
\begin{proof}
The proof is based on three facts. One, the primal~\eqref{primal sdp} and dual~\eqref{dual sdp}  have descriptions of size polynomial in $n$, the number of images.
Two, SDPs can be solved to arbitrary precision in time which is polynomial in the size of their descriptions~\cite{vandenberghe-boyd-1996}. Three, the eigenvalue decomposition of a matrix can be computed in polynomial time.
\qed
\end{proof}


Before moving forward, we need the following definition.
\begin{definition}
A triangulation problem is \sdpexact{} if $\gtwo=\gsos=\gmom$, i.e., the relaxations are tight.
\end{definition}

We will first describe the conditions under which a triangulation problem is \sdpexact. We will then show that if Algorithm~\ref{alg:triangulation} returns {\sf OPTIMAL} then triangulation is \sdpexact, and further, the solution $X$ returned by the algorithm is indeed optimal for triangulation.

\begin{theorem}
\label{thm:sdp-exactness}
Let $x^{*}$ be an optimal solution to the quadratic program~\eqref{eq:matrix-polynomial}.  The triangulation problem is \sdpexact{} if and only if there exist $\lambda_{ij}\in\mathbb{R}$  such that
\begin{equation}
{\rm(i)}\ \nabla g(x^{*}) + \sum\lambda_{ij}\nabla f_{ij}(x^{*}) = 0
\qquad\text{and}\qquad {\rm(ii)}\ I+\sum\lambda_{ij}H_{ij}\succeq 0.
\end{equation}
\end{theorem}

Before proving this theorem, we observe that it is not immediately useful from a computational perspective.  Indeed, {\em a priori} verifying condition (i) requires knowledge of the optimal solution $x^{*}$. However, the theorem will help us understand {\em why} the triangulation problem is so often \sdpexact{} in Section~\ref{sec:geometry of algorithm}.


Let 
$L(x,\lambda_{ij},\rho)  = g(x) + \sum_{ij}\lambda_{ij} f_{ij}(x) - \rho = \tilde{x}^{\top}\left(G+\sum\lambda_{ij}F_{ij} - \rho E\right)\tilde{x}$.
Observe that $\nabla_{x}L(x,\lambda_{ij},\rho) = \nabla g(x) + \sum\lambda_{ij}\nabla f_{ij}(x)$ and $\nabla^{2}_{x}L(x,\lambda_{ij},\rho) = 2(I + \sum \lambda_{ij} H_{ij})$. We require the following two simple lemmas, the proofs of which can be found in the Appendix.
\begin{lemma}
If $x^{*}$ is the optimal solution to~\eqref{eq:matrix-polynomial} and  $\lambda_{ij}$ satisfy condition (i), then 
$L(x,\lambda_{ij}, g(x^{*})) = \left(x-x^{*}\right)^{\top}\left(I+\sum\lambda_{ij}H_{ij}\right)\left(x-x^{*}\right)$. Further, if condition (ii) is satisfied as well, then $L(x,\lambda_{ij}, g(x^{*})) \ge 0,\ \forall x \in \mathbb{R}^{2n}$.
\label{lem:lagrange-sos}
\end{lemma}
\begin{lemma}
If $x^{\top} A x + 2b^{\top} x + c \geq 0,\ \forall x$, then 
$\begin{bmatrix}
A & b\\
b^{\top} & c
\end{bmatrix} 
\succeq 0$.
\label{lem:homogeneous}
\end{lemma}

\begin{proof}[Theorem~\ref{thm:sdp-exactness}]
For the {\em if} direction, let $\lambda_{ij}$ satisfy conditions (i) and (ii). Then from Lemma~\ref{lem:lagrange-sos} we have $L(x,\lambda_{ij},g(x^{*})) \geq 0\ \forall x \in \reals^{2n}$, which combined with Lemma~\ref{lem:homogeneous} gives $G+\sum\lambda_{ij}F_{ij} - g(x^{*})E \succeq 0$. Therefore $\lambda_{ij}$ and $\rho = g(x^{*})$ are dual feasible, which in turn means  that $\gsos \ge \rho=g(x^{*}) = \gtwo$. Lemmas~\ref{lem:mom-relax-inequality} and \ref{lem:zero-duality} give the reverse inequality, thus $\gsos=\gmom=\gtwo = g(x^{*})$.

For the {\em only if} direction, let $\rho^{*}$ and $\lambda^{*}_{ij}$ be the optimal solution to the dual~\eqref{dual sdp}. If the problem is \sdpexact{} we have the equality $\rho^{*} = g(x^{*}) = \gtwo$ and from the dual feasibility of $\lambda^{*}_{ij}$ and $\rho^{*}$ we have that $G+\sum\lambda^{*}_{ij}F_{ij} - \rho^{*}E \succeq 0$. Taken together, these two facts imply that $L(x, \lambda^{*}_{ij}, g(x^{*})) \geq 0,\ \forall x \in \reals^{2n}$.

Since $L(x^{*}, \lambda^{*}_{ij}, g(x^{*})) = 0$, $L(x, \lambda^{*}_{ij}, g(x^{*}))$ is a non-negative quadratic polynomial that vanishes at $x^{*}$. Non-negativity implies condition (ii) (that the Hessian of the polynomial is positive semidefinite) and the fact that zero is the minimum possible value of a non-negative polynomial implies that its gradient vanishes at $x^{*}$ which is exactly condition (i). \qed
\end{proof}


Condition (ii) of Theorem~\ref{thm:sdp-exactness} is automatically satisfied by {\em any} feasible $\lambda_{ij}$ for the dual~\eqref{dual sdp}.  Hence, verifying SDP exactness using the dual optimal $\lambda_{ij}^*$ reduces to checking condition (i) of Theorem \ref{thm:sdp-exactness}.  Checking this condition is not computationally practical, since it requires knowledge of the optimum $x^*$. By slightly tightening condition (ii) we can bypass condition (i).


\begin{theorem}
\label{thm:extract solution}
If  $\{Y^{*},\ \lambda_{ij}^{*},\ \rho^{*}\}$ are primal-dual optimal and
$I+\sum\lambda_{ij}^{*}H_{ij}\succ 0$, then $\operatorname{rank}(Y^{*}) = 1$, and triangulation is \sdpexact{}.
\end{theorem}

\begin{proof}
Notice that $I+\sum\lambda_{ij}^{*}H_{ij}$ is the top left $(2n)\times(2n)$ block in the larger $(2n+1)\times(2n+1)$ positive semidefinite matrix $G+\sum\lambda_{ij}^{*}F_{ij}-\rho^{*} E$.  By hypothesis, $I+\sum\lambda_{ij}^{*}H_{ij}$ is nonsingular and thus has full rank equal to $2n$, which implies
\begin{equation}
\label{eq:rank-2n}
{\rm rank}\left(G+\sum\lambda_{ij}^{*}F_{ij}-\rho^{*} E\right)\ge 2n.
\end{equation}

The dual and the primal SDP solutions satisfy {\em complementary slackness}, which means that $\left\langle G+\sum\lambda_{ij}^{*}F_{ij}-\rho^{*} E,Y^{*}\right\rangle = 0$. In particular it implies that 
\begin{equation}
{\rm rank}(G+\sum\lambda_{ij}^{*}F_{ij}-\rho^{*} E)+{\rm rank}(Y^{*})\le 2n+1,
\label{eq:rank-2n+1}
\end{equation}
where we use the standard fact that whenever $\langle A,B\rangle =0$ for $A,B\in\mathcal{S}^{N}_{+}$, then ${\rm rank}(A)+{\rm rank}(B)\le N$.
From~\eqref{eq:rank-2n} and~\eqref{eq:rank-2n+1} we have ${\rm rank}(Y^{*})\le 1$.  Since $\langle E,Y\rangle = 1$, we have $Y^*\ne 0$ and hence
${\rm rank}(Y^{*})=1$.
\qed
\end{proof}

Line 4 of Algorithm~\ref{alg:triangulation} uses Theorem~\ref{thm:extract solution} to establish that we have solved~\eqref{eq:matrix-polynomial}. Lines 5--10 of the algorithm are then devoted to making sure that the solution actually lies in $\overline{V}_P$. Thus, Algorithm~\ref{alg:triangulation} is correct.

\subsection{Implications}

\begin{theorem}
\label{thm:noiseless}
If the image observations are noise free, i.e. there exists $X^{*} \in \reals^{3}$ such that $\widehat{x}_{i} = \Pi P_{i} X^{*},\ \forall i = 1,\ldots, n$, then Algorithm~\ref{alg:triangulation} returns {\sf OPTIMAL}
\end{theorem}
\begin{proof}
Setting $\lambda_{ij} = 0$ satisfies the hypothesis for Theorems~\ref{thm:sdp-exactness} and~\ref{thm:extract solution}.\qed
\end{proof}


\begin{theorem}
\label{thm:two-view-exact}
Two view triangulation is \sdpexact{}.
\end{theorem}
\begin{proof}
For $n=2$ the triangulation problem~\eqref{eq:matrix-polynomial} involves minimizing a quadratic objective over a single quadratic equality constraint $f_{12}(x)=0$. The conditions of Theorem~\ref{thm:sdp-exactness} in this case reduce to finding $\lambda\in\mathbb{R}$ satisfying 
\begin{equation}
\nabla g(x^{*}) + \lambda \nabla f_{12}(x^{*}) = 0\qquad\text{and}\qquad I+\lambda H_{12}  \succeq 0. \label{eq:twoview}
\end{equation}
The existence of such a $\lambda$ follows directly from \cite[Theorem 3.2]{more-1993}.\qed
\end{proof}

Theorems~\ref{thm:extract solution} and~\ref{thm:two-view-exact} nearly imply that two-view triangulation can be solved in polynomial time. We say {\em nearly} because, despite Theorem~\ref{thm:two-view-exact}, it is possible that the matrix $I+\lambda^* H_{12}$ is singular for the dual optimal $\lambda^*$ (see Appendix for an example). This is not a contradiction, since Theorem~\ref{thm:extract solution} is only a sufficient condition for optimality.  Despite such pathologies, we shall see in Section~\ref{sec:experiments} that in practice Algorithm~\ref{alg:triangulation} usually returns  {\sf OPTIMAL} for two-view triangulation.


\subsection{Geometry of the Algorithm}
\label{sec:geometry of algorithm}
Recall that the optimization problem~\eqref{eq:matrix-polynomial} can be interpreted as determining the closest point $x^*$ to $\widehat{x}$ in the variety $W_P$. This viewpoint gives geometric intuition for why Algorithm~\ref{alg:triangulation} can be expected to perform well in practice. 

\begin{lemma}
\label{lem:gradient-slice}
Given $\widehat{x}$, and assuming appropriate regularity conditions at the optimal solution $x^{*}$ to~\eqref{eq:matrix-polynomial}, 
there exist $\lambda_{ij} \in \reals$ such that 
\begin{equation}
\label{eq:in-gradient-slice}
\widehat x = x^* + \sum\frac{\lambda_{ij}}{2}\nabla f_{ij}(x^*)
\end{equation}
\end{lemma}
\begin{proof}
It follows from Lagrange multiplier theory~\cite{nocedal-wright-1999} that there exist Lagrange multipliers $\lambda_{ij}\in\mathbb{R}$ such that $\nabla g(x^*)+\sum\lambda_{ij}\nabla f_{ij}(x^*)=0$.  Observe that for~\eqref{eq:matrix-polynomial} $\nabla g(x^*)=2(x^*-\widehat x)$, which finishes the proof. \qed
\end{proof}

If $\|x^*-\widehat{x}\|$ is small, i.e. $\widehat x$ is close to the variety $W_{P}$, then there must exist some $\lambda_{ij}$ satisfying \eqref{eq:in-gradient-slice} such that $\|\lambda\|$ is small and hence $I+\sum\lambda_{ij}H_{ij}$ is a small perturbation of the positive definite identity matrix $I$. Since $I$ lies in the interior of the positive semidefinite cone, these small perturbations also lie in the interior, that is $I+\sum\lambda_{ij}H_{ij}\succ 0$.

This, coupled with the fact that $\lambda_{ij}$ are Lagrange multipliers at $x^{*}$, yields the sufficient conditions in Theorem~\ref{thm:sdp-exactness} for a triangulation problem to be \sdpexact{}. Thus, if the amount of noise in the observations is small, Algorithm~\ref{alg:triangulation} can be expected to recover the optimal solution to the triangulation problem.
Since $H_{ij}$ depends only on the cameras $P_{i}$ and not on $\widehat x$, 
the amount of noise which Algorithm~\ref{alg:triangulation} will tolerate depends only on $P_{i}$.
We summarize this formally in the following theorem.



\begin{theorem}
\label{thm:neighborhood}
Let $N(x^{*}) = \{x^* + \sum\frac{\lambda_{ij}}{2}\nabla f_{ij}(x^*)\ |\ I+\sum\lambda_{ij}H_{ij}\succ 0\}$. For any $\widehat x\in\reals^{2n}$,  if Algorithm~\ref{alg:triangulation} returns {\sf OPTIMAL} and $x^{*}$ is the optimal image projection vector  then  $\widehat x\in N(x^{*})$.  Conversely, if $\widehat{x}\in N(x^{*})$ and $x^{*}$ is the closest point in $W_{P}$ to $\widehat x$, then Algorithm~\ref{alg:triangulation} will return {\sf OPTIMAL}.
\end{theorem}
\begin{proof}
The proof is a straightforward application of Lemma~\ref{lem:gradient-slice} and Theorem~\ref{thm:extract solution}.\qed
\end{proof}

\section{Experiments}
\label{sec:experiments}

Algorithm~\ref{alg:triangulation} was implemented using {\tt YALMIP}~\cite{yalmip}, {\tt SeDuMi}~\cite{sedumi} and {\tt MATLAB}.  These tools allow for easy implementation and the timings below are for completeness  and should not be used to judge the runtime performance of the algorithm.

Fundamental matrices and epipolar constraints are specified only up to scale and badly scaled fundamental matrices lead to poorly conditioned SDPs. This was easily fixed by dividing each $F_{ij}$ by its largest singular value.

Algorithm~\ref{alg:triangulation} verifies optimality by checking the positive definiteness of $I+\sum\lambda_{ij}^{*}H_{ij}\succ 0$, which requires that its smallest eigenvalue be greater than some $\delta > 0$. In all our experiments we set $\delta = 0.05$.

Either Algorithm~\ref{alg:triangulation} returns {\sf OPTIMAL} and the solution is guaranteed to be optimal, or it returns {\sf SUBOPTIMAL} in which case we cannot say anything about the quality of the solution,  even though it could still be optimal or at least serve as a good guess for non-linear iteration. 
Thus, we run Algorithm~\ref{alg:triangulation} on a number of synthetic and real-world datasets and report the fraction of cases in which the algorithm certifies optimality.


\subsection{Synthetic Data}

To test the performance of the algorithm as a function of camera configuration and image noise we first tested Algorithm~\ref{alg:triangulation} on three synthetic data sets. Following~\cite{Olsson-el-al-2009}, we created instances of the triangulation problem by randomly generating points inside the unit cube in $\mathbb{R}^3$.  

For the first experiment, a varying number of cameras (2, 3, 5 and 7) were placed uniformly at random on a sphere of radius 2. In the second experiment, the same number of cameras were uniformly distributed on a circle of radius 2 in the $xy$-plane. In the third experiment they were restricted to the $x$-axis and were placed at a distance of $3$, $5$, $7$, and $9$ units. These setups result in image measurements with an approximate
square length of 2 units.  Gaussian noise of varying standard deviation
was added to the image measurements. The maximum standard deviation was $0.2$, which corresponds to about $10\%$ of the image size. For each noise level we ran the experiment 375 times. Figures~\ref{fig:synthetic graph}(a), (b) and (c) show the fraction of test cases for  which Algorithm~\ref{alg:triangulation} returned {\sf OPTIMAL} as a function of the the standard deviation of perturbation noise.

\begin{figure}[t]
\begin{center}
\begin{tabular}{ccc}
\includegraphics[width=0.33\linewidth]{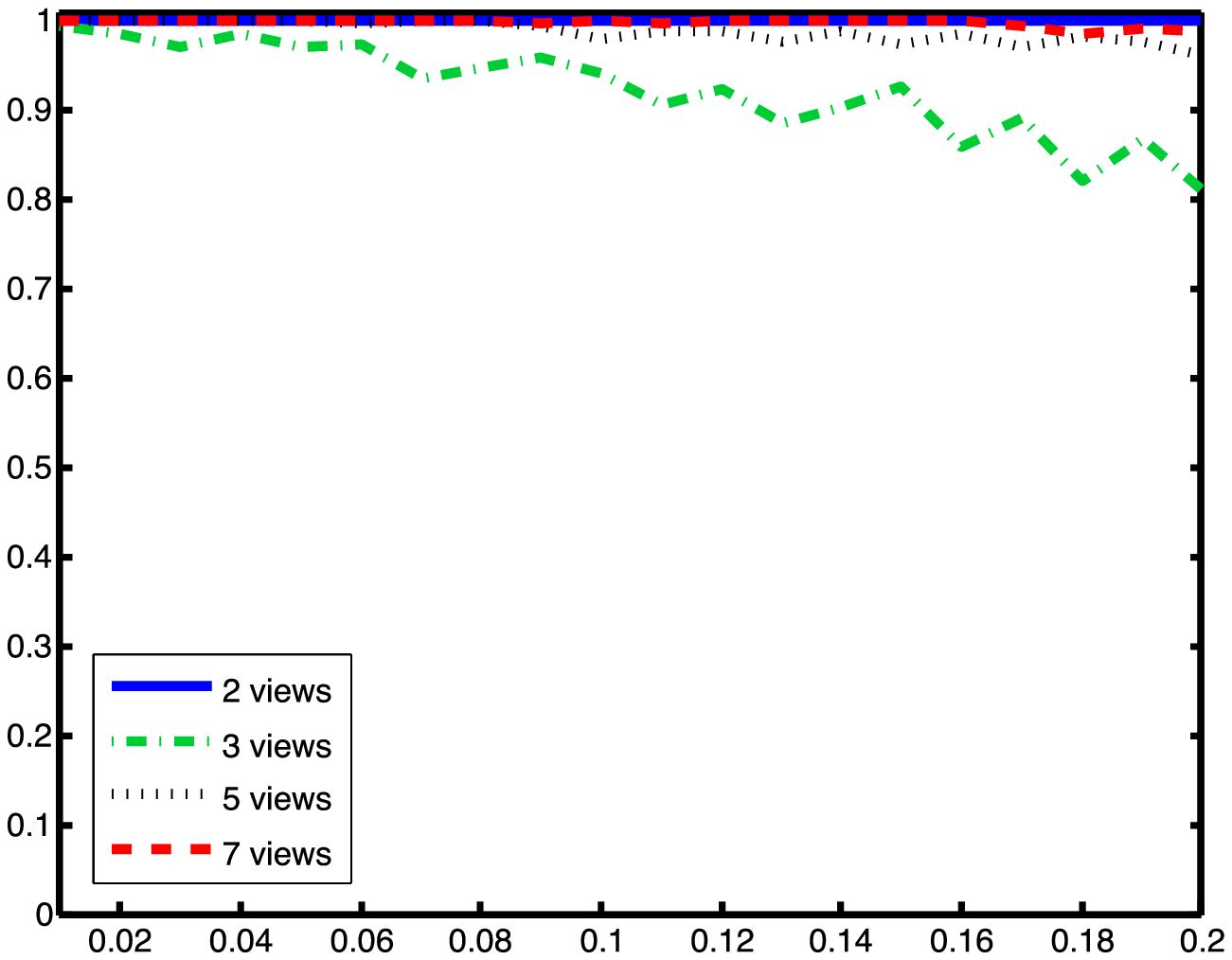} & 
\includegraphics[width=0.33\linewidth]{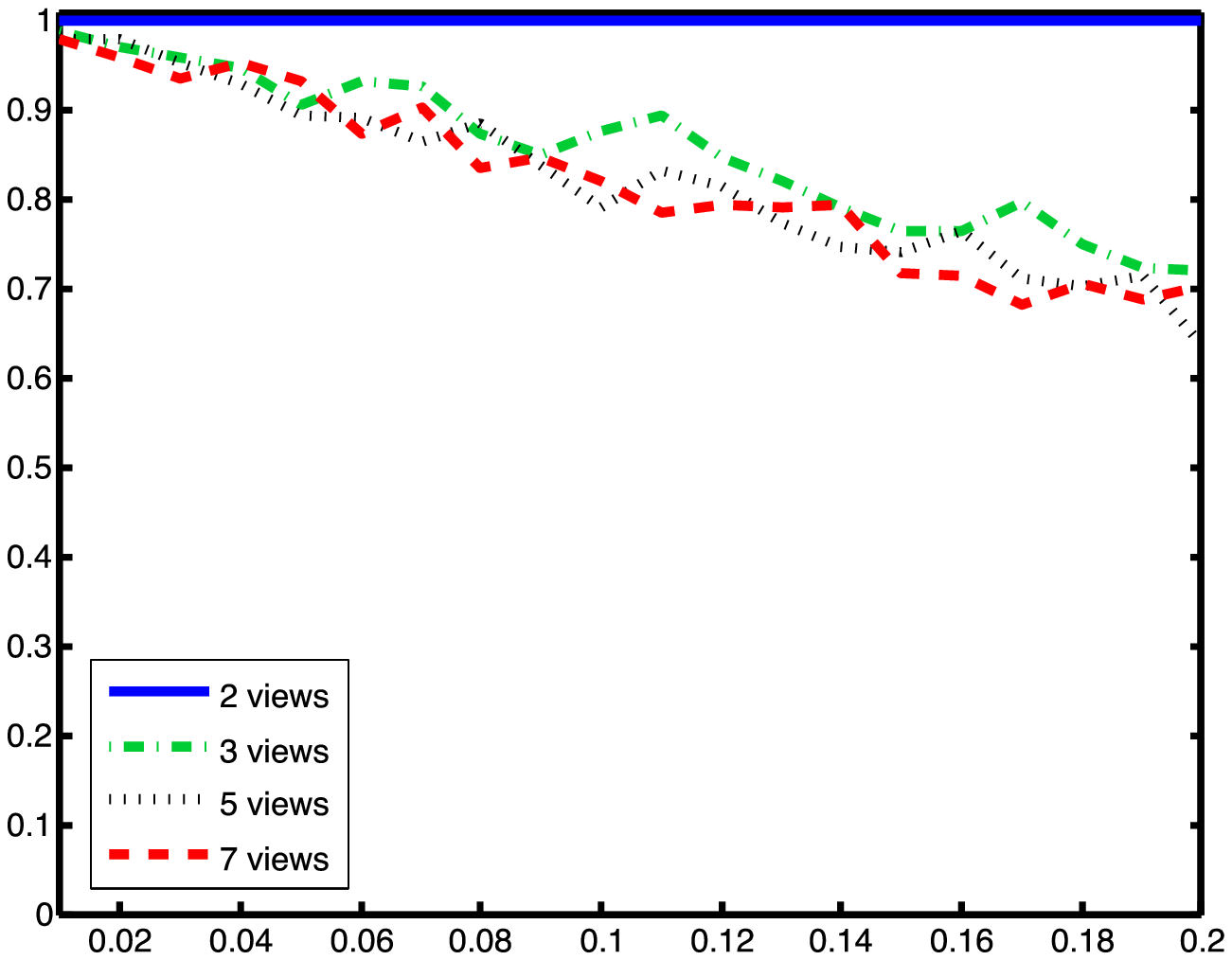} &
\includegraphics[width=0.33\linewidth]{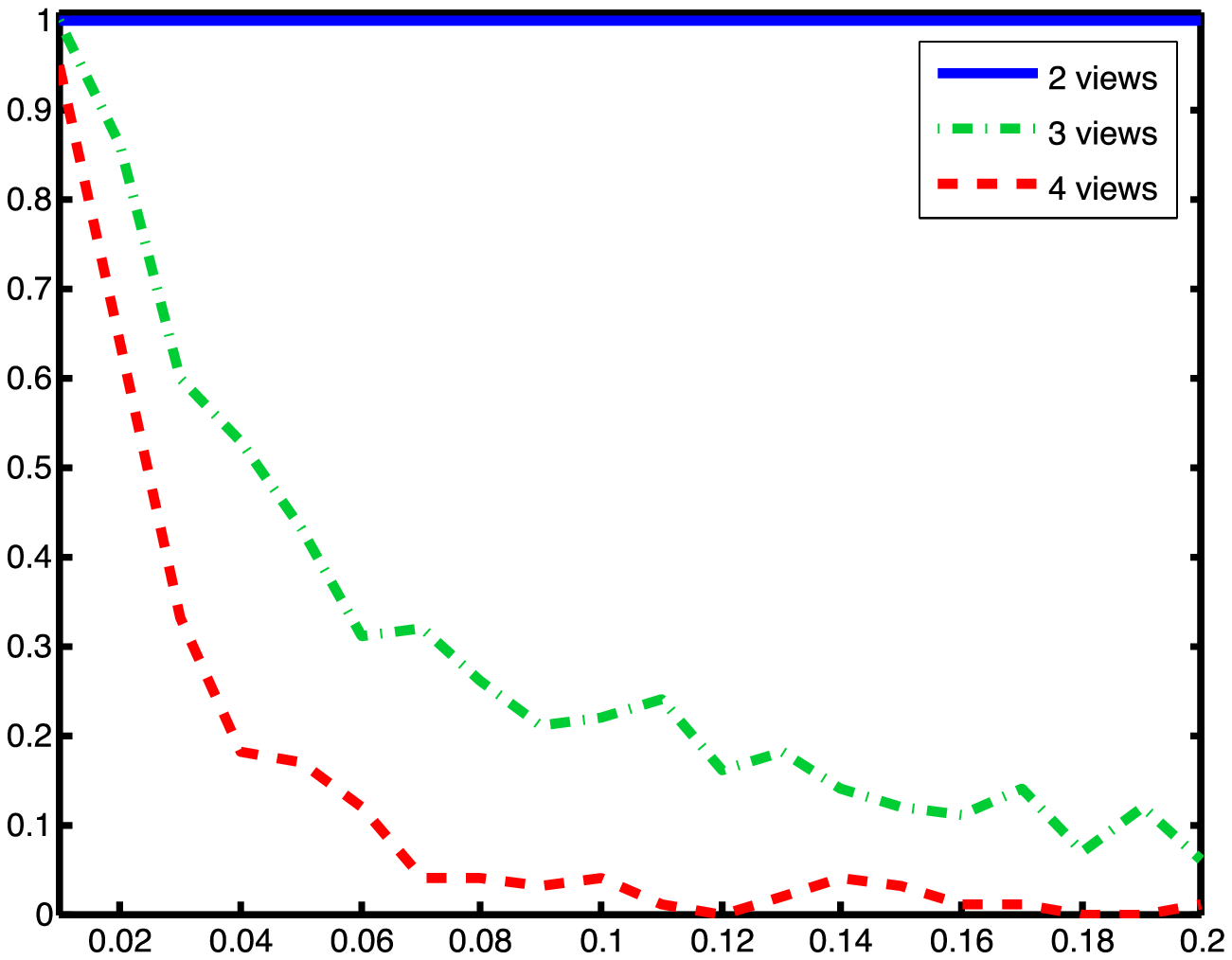}\\
(a) Cameras on a sphere. & (b) Cameras on a circle. & (c) Cameras on a line.
\end{tabular}
\end{center}
\caption{Fraction of synthetic experiments in which Algorithm~\ref{alg:triangulation}
returned {\sf OPTIMAL} versus the standard deviation
of the noise level added to the images. (a) Cameras
placed randomly on the sphere of radius 2.  (b) Three cameras
all placed on the $xy$-plane. (c) Three cameras placed along the $x$-axis.}
\label{fig:synthetic graph}
\end{figure}

We first note that for $n=2$ cameras, in all three cases, independent of camera geometry and noise, we are able to solve the triangulation problem optimally $100\%$ of the time. This experimentally validates Theorem~\ref{thm:two-view-exact} and provides strong evidence that for $n=2$ there is no gap between Theorems~\ref{thm:sdp-exactness} and~\ref{thm:extract solution} in practice.

Observe that for cameras on a sphere (Figure~\ref{fig:synthetic graph}(a)), the algorithm performs very well, and while the performance drops as the noise is increased, the drop is not significant for practical noise levels. Another interesting feature of this graph is that finding and certifying optimality becomes easier with increasing number of cameras. The reason for this increase in robustness is not clear and we plan to further pursue this intriguing  observation in future research.

Cameras on a circle (Figure~\ref{fig:synthetic graph}(b)) is one of the degenerate cases of our algorithm where  $\overline{V}_{P} \ne W_{P}$. We expect a higher rate of failure here and indeed this is the case. It is however worth noting that the algorithm does not immediately breakdown and shows a steady decline in performance as a function of noise like the previous experiment. Unlike cameras on a sphere, increasing the number of cameras does not increase the performance of the algorithm, which points to the non-trivial gap between $\overline{V}_{P}$ and $W_{P}$ for co-planar cameras.

Finally let us consider the hard problem of triangulating a point when the camera is moving on a line. This case is hard geometrically and algorithmically as the cameras are co-planar, and this difficulty is reflected in the performance of the algorithm. In contrast to the previous experiments, we observe rapid degradation with increasing noise. 

\subsection{Real Data}

We tested Algorithm~\ref{alg:triangulation} on four real-world data sets: the {\sf Model House}, {\sf Corridor}, and {\sf Dinosaur} from Oxford University and the {\sf Notre Dame} data set~\cite{snavely-et-al-2006}. Our results are summarized in Table \ref{tab:real data}.
\begin{table}
\renewcommand{\tabcolsep}{0.2cm}
\begin{center}
\caption{Performance of Algorithm~\ref{alg:triangulation} on real data. The column {\sf OPTIMAL} reports the fraction of triangulation problems which were certified optimal.}
\label{tab:real data}
\begin{tabular}{lrrrr}
\toprule
Data set & \# images & \# points & {\sf OPTIMAL} & Time (sec)\\
\midrule
{\sf Model House} & 10 & 672 & 1.000 & 143\\
{\sf Corridor} & 11 & 737 & 0.999 & 193\\
{\sf Dinosaur} & 36 & 4983 & 1.000& 960\\
{\sf Notre Dame} & 48 & 16,288 & 0.984 & 7200\\
\bottomrule
\end{tabular}
\end{center}
\end{table}

{\sf Model house} is a simple data set where the camera moves 
laterally in front of the model house.
Global optimality was achieved in all cases.

{\sf Corridor} is a geometrically hard sequence where most of the camera motion is
forward, straight down a corridor. This is similar to the synthetic experiment
where the cameras were all on the $x$-axis. The algorithm returned {\sf OPTIMAL} in all but one case.

{\sf Dinosaur} consists of images of a (fake) dinosaur rotating on a turntable in front of a fixed camera. Even though the camera configuration is hard for our algorithm (cameras are co-planar),  global optimality was achieved in all cases. 

The three Oxford datasets are custom captures from the same camera. {\sf Notre Dame} consists of images of the Notre Dame cathedral downloaded from {\sf Flickr}. The data set comes with estimates of the radial distortion for each camera, which we ignore as we are only considering projective cameras in this paper. Algorithm~\ref{alg:triangulation}
returned {\sf OPTIMAL} in $98.36\%$ of cases.

It is worth noting here that  the synthetic datasets  were designed to test the limits of the algorithm as a function of noise.  A standard deviation equal to $0.2$ translates to image noise of approximately $10\%$ of the image size. In practice such high levels of noise rarely occur, and when they do, they typically correspond to outliers. The results in Table~\ref{tab:real data} indicate that the algorithm has excellent performance in practice.


\section{Discussion}
\label{sec:discussion}

We have presented a semidefinite programming algorithm for triangulation.
In practice it usually returns a global optimum and a certificate of optimality in polynomial time. 
Of course there are downsides which must be taken into consideration.
Solving SDPs is not nearly as fast as gradient descent methods.
Moreover, what happens in the rare cases that global optimality
is not certifiable? Regardless, the lure of a polynomial time algorithm and a better understanding of the geometry of the triangulation problem is far
too great to allow these hiccups to close the door on SDPs.

Hartley and Sturm \cite{hartley-sturm-1997} show that two-view
triangulation can be solved by finding the roots of a degree 6 univariate polynomial.
For $n=3$,  \cite{stewenius-et-al-2005} gives a Gr\"obner basis based algorithm for triangulation by finding all solutions to a polynomial system with 
$47$ roots. These methods do not extend in a computationally useful manner to $n \geq 4$. Our algorithm solves the triangulation problem for all values of $n \geq 2$ under the conditions of Theorem~\ref{thm:extract solution} when the camera centers are not co-planar. 

Similar to our setup, Kanatani et al. also frame triangulation as finding the closest point on a variety from noisy observations~\cite{kanatani-et-al-2010}. Unlike our work, they use both epipolar and trifocal constraints, which makes the constraint set much more complicated. Further, they do not prove finite convergence or the optimality of their procedure. Our work answers their question of analyzing the shape of the variety to obtain a {\em noise threshold} for optimality guarantees.

Semidefinite relaxations for optimization problems in multi-view geometry  were first studied by  Kahl \& Henrion in \cite{kahl-henrion-2007}.  They formulate triangulation as an optimization problem over quartic inequalities and study its moment relaxations. Kahl \& Henrion observe that good solutions can be obtained from the first few relaxations. In our method the very first relaxation already yields high quality solutions.
Further, the quadratic equality based formulation has nice theoretical properties that explain the empirical performance of our algorithm.

Hartley \& Seo  proposed a method for verifying global optimality of a given solution to the triangulation problem~\cite{hartley-seo-2008}. The relation between their test and the definiteness test of Theorem~\ref{thm:extract solution} is a fascinating direction for future research.

The study of QCQPs has a long history. There is a wide body of work devoted to verifying optimality of a given solution for various classes of QCQPs and using semidefinite programming to solve them. Thus it is natural that statements similar to Theorems~\ref{thm:sdp-exactness} and \ref{thm:extract solution}  for various subclasses of QCQPs have appeared several times in the  past e.g.~\cite{jeyakumar-et-al-2007,li-2012,pinar-2004}.

The astute reader will notice that even though Theorems~\ref{thm:sdp-exactness} and \ref{thm:extract solution} are presented for triangulation, the proofs apply to any QCQP with equality constraints, and can be interpreted in terms of the Karush-Kuhn-Tucker (KKT) conditions for~\eqref{eq:matrix-polynomial}. More recently, and independently from us, Zheng et al. have proved versions of Theorems~\ref{thm:sdp-exactness} and \ref{thm:extract solution} for the more general case of inequality-constrained QCQPs~\cite{zheng-et-al-2012}. Our formulation of triangulation as minimizing distance to a variety allows for a geometric interpretation of these optimality conditions, which in turn explains the performance of our algorithm.

Finally, we use the general purpose SDP solver {\tt SeDuMi} for our experiments. Much improvement can be expected from a custom SDP solver which makes use of the explicit symmetry and sparsity of the triangulation problem (e.g. for all $n$ the linear matrix  constraints in the primal semidefinite program~\eqref{primal sdp} are sparse and have rank at most five).

\subsection*{Acknowledgements}
It is a pleasure to acknowledge our conversations with  Pablo Parrilo, Jim Burke, Greg Blekherman and Fredrik Kahl. In particular Fredrik Kahl introduced us to the work of Heyden \& {\AA}str{\"o}m~\cite{heyden-astrom-1997}. Chris Aholt and Rekha Thomas were supported by the NSF grant DMS-1115293.


\bibliographystyle{abbrv}
\bibliography{eccv2012_sdp}

\begin{thebibliography}{10}

\bibitem{freund-jarre-2001}
R.~W. Freund and F.~Jarre.
\newblock Solving the sum-of-ratios problem by an interior-point method.
\newblock {\em J. Glob. Opt.}, 19(1):83--102, 2001.

\bibitem{hartley-seo-2008}
R.~Hartley and Y.~Seo.
\newblock Verifying global minima for $l_2$ minimization problems.
\newblock In {\em CVPR}, 2008.

\bibitem{hartley-sturm-1997}
R.~Hartley and P.~Sturm.
\newblock Triangulation.
\newblock {\em CVIU}, 68(2):146--157, 1997.

\bibitem{hartley-zisserman-2003}
R.~Hartley and A.~Zisserman.
\newblock {\em Multiple View Geometry in Computer Vision}.
\newblock Cambridge University Press, second edition, 2003.

\bibitem{heyden-astrom-1997}
A.~Heyden and K.~{\AA}str{\"o}m.
\newblock Algebraic properties of multilinear constraints.
\newblock {\em Math. Methods Appl. Sci.}, 20(13):1135--1162, 1997.

\bibitem{jeyakumar-et-al-2007}
V.~Jeyakumar, A.~Rubinov, and Z.~Wu.
\newblock Non-convex quadratic minimization problems with quadratic
  constraints: global optimality conditions.
\newblock {\em Mathematical Programming}, 110(3):521--541, 2007.

\bibitem{kahl-et-al-2008}
F.~Kahl, S.~Agarwal, M.~K. Chandraker, D.~J. Kriegman, and S.~Belongie.
\newblock Practical global optimization for multiview geometry.
\newblock {\em IJCV}, 79(3):271--284, 2008.

\bibitem{kahl-henrion-2007}
F.~Kahl and D.~Henrion.
\newblock Globally optimal estimates for geometric reconstruction problems.
\newblock {\em IJCV}, 74(1):3--15, 2007.

\bibitem{kanatani-et-al-2010}
K.~Kanatani, H.~Niitsuma, and Y.~Sugaya.
\newblock Optimization without search: Constraint satisfaction by orthogonal
  projection with applications to multiview triangulation.
\newblock {\em Memoirs of the Faculty of Engineering}, 44:32--41, 2010.

\bibitem{laurent-2009}
M.~Laurent.
\newblock Sums of squares, moment matrices and optimization over polynomials.
\newblock volume 149 of {\em IMA Vol. Math. Appl.}, pages 157--270. Springer,
  2009.

\bibitem{li-2012}
G.~Li.
\newblock Global quadratic minimization over bivalent constraints: Necessary
  and sufficient global optimality condition.
\newblock {\em Journal of Optimization Theory and Applications}, pages 1--17,
  2012.

\bibitem{yalmip}
J.~Lofberg.
\newblock {YALMIP}: A toolbox for modeling and optimization in matlab.
\newblock In {\em Int. Symp. on Computer Aided Control Systems Design}, pages
  284--289, 2004.

\bibitem{lu-hartley-2007}
F.~Lu and R.~I. Hartley.
\newblock A fast optimal algorithm for $l_2$ triangulation.
\newblock In {\em ACCV}, volume~2, pages 279--288, 2007.

\bibitem{more-1993}
J.~Mor\'e.
\newblock Generalizations of the trust region problem.
\newblock {\em Optimization methods and Software}, 2(3-4):189--209, 1993.

\bibitem{nocedal-wright-1999}
J.~Nocedal and S.~Wright.
\newblock {\em Numerical Optimization}.
\newblock Springer Verlag, 1999.

\bibitem{Olsson-el-al-2009}
C.~Olsson, F.~Kahl, and R.~Hartley.
\newblock Projective least-squares: Global solutions with local optimization.
\newblock In {\em CVPR}, pages 1216--1223, 2009.

\bibitem{pinar-2004}
M.~Pinar.
\newblock Sufficient global optimality conditions for bivalent quadratic
  optimization.
\newblock {\em Journal of optimization theory and applications},
  122(2):433--440, 2004.

\bibitem{shafarevich-1988}
I.~Shafarevich.
\newblock {\em Basic Algebraic Geometry I: Varieties in Projective Space}.
\newblock Springer-Verlag, 1998.

\bibitem{snavely-et-al-2006}
N.~Snavely, S.~M. Seitz, and R.~Szeliski.
\newblock Photo tourism: Exploring photo collections in 3d.
\newblock In {\em SIGGRAPH}, pages 835--846, 2006.

\bibitem{stewenius-et-al-2005}
H.~Stewenius, F.~Schaffalitzky, and D.~Nister.
\newblock How hard is 3-view triangulation really?
\newblock In {\em ICCV}, pages 686--693, 2005.

\bibitem{sedumi}
J.~Sturm.
\newblock Using {SeDuMi} 1.02, a {M}atlab toolbox for optimization over
  symmetric cones.
\newblock {\em Opt. Meth. and Soft.}, 11-12:625--653, 1999.

\bibitem{vandenberghe-boyd-1996}
L.~Vandenberghe and S.~Boyd.
\newblock Semidefinite programming.
\newblock {\em SIAM Review}, 38(1):49--95, 1996.

\bibitem{zheng-et-al-2012}
X.~Zheng, X.~Sun, D.~Li, and Y.~Xu.
\newblock On zero duality gap in nonconvex quadratic programming problems.
\newblock {\em Journal of Global Optimization}, 52(2):229--242, 2012.

\end{thebibliography}


\section*{Appendix}

\subsection*{Failure of Theorem~\ref{thm:extract solution} for two-view triangulation}
Consider the cameras
\begin{equation}
A_1 = \begin{bmatrix}0&0&1&0\\0&1&0&0\\-1&0&0&a\end{bmatrix}\qquad
A_2 = \begin{bmatrix}0&0&1&0\\0&1&0&0\\-1&0&0&b\end{bmatrix},
\end{equation}
where $0<a<b$.  This models the situation in which two cameras are placed on the $x$-axis at $(a,0,0)$ and $(b,0,0)$, respectively, both facing the origin (i.e. viewing down the $x$-axis).

We can show that, given $\varepsilon>0$ and $\widehat{x}=(\widehat{x}_1,\widehat{x}_2)=\left(\begin{bmatrix}0;\varepsilon\end{bmatrix},\begin{bmatrix}\varepsilon;0\end{bmatrix}\right)$, the minimum value of the QCQP~\eqref{eq:matrix-polynomial} is $\varepsilon^2$ and all points of the form 
\begin{align}
    x_1^* &= \begin{bmatrix} \sqrt{\mu(\varepsilon -\mu)}; & \mu \end{bmatrix}\\
    x_2^* &= \begin{bmatrix} \varepsilon - \mu; & \sqrt{\mu(\varepsilon -\mu)}\end{bmatrix}
\end{align}
where $0\le \mu\le \varepsilon$ are optimal.


The non-uniqueness of the minimizer implies that the verification matrix $I+\lambda H_{12}$ cannot be positive definite. So in this case, the hypothesis of Theorem~\ref{thm:extract solution} is not satisfied, although Theorem~\ref{thm:two-view-exact} still holds.

The proof of these facts is a straightforward application of the KKT conditions. The minimum distance $\varepsilon^2$ can be attained by setting one of the image points equal to the epipole, but not both. However, the minimum distance is also attained at points which correspond to neither image being the epipole.

\subsection*{Proof of Lemma~\ref{lem:lagrange-sos}}
Recall conditions (i) $\nabla g(x^{*})+\sum\lambda_{ij}\nabla f_{ij}(x^{*})=0$, and (ii) $I+\sum\lambda_{ij}H_{ij}\succeq 0$.  Let $H=\sum\lambda_{ij}H_{ij}$, $b=\sum\lambda_{ij}b_{ij}$, and $\beta=\sum\lambda_{ij}\beta_{ij}$. Then
\begin{align}
L(x,\lambda_{ij},g(x^*)) &= g(x)+\sum\lambda_{ij}f_{ij}(x)-g(x^*)\\
&= x^{\top}\left(I+H\right)x + 2\left(b-\widehat x\right)^{\top}x + \widehat x^{\top}\widehat x + \beta -g(x^{*}).\label{eq:long-L}
\end{align}
We wish to show that~\eqref{eq:long-L} is equal to $\left(x-x^{*}\right)^{\top}\left(I+H\right)\left(x-x^{*}\right)$.
The fact that $L(x,\lambda_{ij},g(x^*))\ge 0$ is then immediate because $I+H\succeq 0$.
Looking at the quadratic, linear, and constant terms separately, it suffices to prove
\begin{equation}
{\rm (a)}\ (I+H)x^* = \widehat x-b\qquad\text{and}\qquad{\rm (b)}\ {x^{*}}^{\top}\left(I+H\right){x^{*}}= \widehat x^{\top}\widehat x + \beta -g(x^{*}).
\end{equation}

Indeed, (a) is just a restatement of condition (i), since $\nabla g(x) = 2(x-\widehat x)$ and $\nabla f_{ij}(x) = 2(H_{ij}x+b_{ij})$. To prove (b), first recall that since $x^{*}$ is feasible for problem~\eqref{eq:matrix-polynomial}, $f_{ij}(x^{*})=0$ for all $i,j$.  In particular, $\beta = -2b^{\top}x^{*} - {x^{*}}^{\top}Hx^{*}.$
Using this and part (a), one can verify (b) through straightforward manipulations.\qed

\subsection*{Proof of Lemma~\ref{lem:homogeneous}}

Suppose the matrix $M = \left[A, b\,; b^{\top}, c\right]$ is not positive semidefinite, i.e. there exists $y$ such that $y^{\top}My<0$. Write $y = [y'; \gamma]$ for some $\gamma\in\mathbb{R}$. If $\gamma=0$, then $0>y^{\top}My={y'}^{\top}Ay'$.  But then we arrive at a contradiction by considering $x=\mu y'$ for a scalar $\mu\in\mathbb{R}$, since for $\mu$ large enough, $x^{\top}Ax+2b^{\top}x+c<0$.

Now if $\gamma\ne 0$, setting $x=y'/\gamma$ gives
\begin{equation}
x^{\top}Ax+2b^{\top}x+c = \frac{1}{\gamma^2}\left(y^{\top}My\right)<0,
\end{equation}
which again contradicts the hypothesis.\qed


\end{document}